\documentclass[letterpaper, 11pt]{amsart}

\usepackage{amsmath,amsthm,amsfonts,amssymb,amscd}
\usepackage{bbm}
\usepackage{bm}
\usepackage{tikz}
\usepackage{tikz-cd}
\usepackage{appendix}
\usepackage{BOONDOX-calo}
\usepackage{standalone}
\usetikzlibrary{arrows,chains,matrix,positioning,scopes, cd}
\usepackage[letterpaper, left=3cm,right=3cm, top=3cm, bottom=3cm]{geometry}
\usepackage{adjustbox}
\usepackage{enumitem}
\usepackage[all]{xy}
\usepackage[
colorlinks=true, citecolor=blue, linkcolor=blue, urlcolor=red]{hyperref}
\usepackage{multicol}
\allowdisplaybreaks
\usepackage{mathrsfs}
\usepackage{here}
\usepackage{blkarray}
\usepackage{caption} 
\usepackage{footnote}
\usetikzlibrary{patterns}
\usepackage{accents}
\usepackage{fvextra, fancyvrb}

\newcommand{\R}{{\mathbb R}}

\newcommand{\Z}{{\mathbb Z}}
\newcommand{\Q}{{\mathbb Q}}

\newcommand{\bE}{{\mathbf E}}

\newcommand\dual{\raise0.9ex\hbox{$\scriptscriptstyle\vee$}}

\newcommand{\mathsym}[1]{{}}
\newcommand{\unicode}[1]{{}}

\newcommand{\dR}{\mathrm{dR}}

\newcommand{\wX}{{\widetilde{X}}}
\newcommand{\wA}{{\widetilde{A}}}
\newcommand{\wB}{{\widetilde{B}}}

\newcommand{\D}{{\Delta}}

\newcommand{\dd}{{\textup{d}}}
\newcommand{\Jh}{{\widehat{J}}}

\theoremstyle{plain}
\newtheorem{thm}{Theorem} 
\newtheorem{prop}[thm]{Proposition}

\newtheorem{cor}[thm]{Corollary}
\numberwithin{thm}{section}
\numberwithin{equation}{section}

\newtheorem{manualtheoreminner}{Theorem}
\newenvironment{thm'}[1]{%
  \renewcommand\themanualtheoreminner{#1}%
  \manualtheoreminner
}{\endmanualtheoreminner}

\theoremstyle{definition}

\theoremstyle{remark}
\newtheorem{rem}[thm]{Remark}

\numberwithin{equation}{section}

\tikzset{>=stealth}

\makeatletter
\def\@seccntformat#1{%
  \protect\textup{\protect\@secnumfont
    \ifnum\pdfstrcmp{subsection}{#1}=0 \bfseries\fi
    \csname the#1\endcsname
    \protect\@secnumpunct
  }%
}  
\makeatother

\makeatletter
\@namedef{subjclassname@2020}{%
  $2020$ Mathematics Subject Classification}
\makeatother

\begin{document}

\title[On two families of period integrals related to Catalan's constant]{On two families of period integrals related to Catalan's constant}
\author{Payman Eskandari}
\address{Department of Mathematics and Statistics, University of Winnipeg, Winnipeg MB, Canada }
\email{p.eskandari@uwinnipeg.ca}
\begin{abstract}
The construction of a mixed motive associated with Catalan's constant $G=\sum_{n=0}^\infty (-1)^n/(2n+1)^2$ in \cite{EMN} led to a supply of period integrals that evaluate to linear forms in 1 and $G$. Inside this supply, the boundary of the integration domain leads to a natural 3-parameter family of period integrals. We show that this family is closely related to the family of hypergeometric period integrals from \cite{RZ} and \cite{Ne}.
\end{abstract}
\maketitle
\vspace*{-.2in}
\section{Introduction}
Catalan's constant
\[
G:=1-\frac{1}{3^2}+\frac{1}{5^2}-\frac{1}{7^2}+\frac{1}{9^2}-\cdots
\]
is one of the most classical constants the irrationality of which is expected but seems to remain open. There are two constructions of period integrals that lead to linear forms in 1 and $G$, one originally due to Rivoal and Zudilin \cite{RZ} which uses values of hypergeometric functions, and another construction coming from \cite{EMN} which has motivic origins. The purpose of this paper is to show that the two constructions are closely related.
\medskip\par 
We start with a brief review of the second construction. Let
\begin{equation}\label{eq1}
    \D = \{(x,y)\in\R^2: x,y,1-x-y\geq 0\}.
\end{equation}
By Theorem 7.2.4 of \cite{EMN}, for every \emph{symmetric} polynomial $F\in \Z[x^2,y^2]$, one has
\[
\int\limits_\D \frac{F}{(1-x^2-y^2)^{k+1}} \,\dd x\dd y \in \Q+\Q G
\]
for any $k\in\Z$ provided that the integrand differential form is ``integrable" (a condition that guarantees convergence of the integral, see \cite[\S 7.2]{EMN}). For applications to rational approximations, one would like to consider a family of integrand functions whose integrals converge to zero rapidly. This naturally leads us to consider integrand functions that vanish on the boundary of $\D$. Taking into account symmetry requirements, this brings us to integrals of the form
\begin{equation}\label{eq2}
    \int\limits_\D \frac{f^\ell h^m}{g^{k+1}} \dd x\dd y,
\end{equation}
where here and throughout,
\begin{equation}\label{eq3}
    f=x^2y^2 ,\quad\quad g=1-x^2-y^2, \quad\quad h=\prod_{\delta,\delta'\in\{\pm 1\}} (1-\delta x -\delta' y).
\end{equation}
The integral \eqref{eq2} converges if and only if 
\begin{equation}\label{eq: original I family convergence cond}
    \ell,m\geq 0, \quad  2\ell+2m\geq k,
\end{equation}
in which case the integrability condition of \cite{EMN} holds and the integral will thus be in $\Q+\Q G$.
\medskip\par 
The other construction uses generalized hypergeometric functions. Traces of the link between values of the generalized hypergeometric function ${}_{3}F_2$ at 1 and $G$ can already be found in work of Nielsen \cite[p. 166]{Nielsen} and Ramanujan \cite[p. 289]{Ram}. In its more mature form, Rivoal and Zudilin \cite[\S 9]{RZ} give a 5-parameter family of ${}_3F_2$-values at 1 which give linear forms in 1 and $G$. This family, which has been further studied more recently by Nesterenko \cite{Ne}, is given by
\begin{equation}\label{eq6}
    \int\limits_{[0,1]^2} \frac{x^{a_1-1/2}y^{a_2}(1-x)^{b_1}(1-y)^{b_2-1/2}}{(1-xy)^{a_3+1}} \dd x \dd y \ \in \ \Q+\Q G,
\end{equation}
where the $a_i, b_j$ are integers satisfying the convergence conditions
\[
a_1,a_2,b_1,b_2\geq 0, \qquad b_1+b_2\geq a_3.
\]

Reflecting the original ideas leading to the two families \eqref{eq2} and \eqref{eq6}, let us refer to the two families respectively as motivic and hypergeometric. Note that this is not to say that the hypergeometric family has no motivic interpretation. The terms are only meant to reflect the origins of the two families. In fact, a change of variables $x=u^2$ and $1-y=v^2$ would rationalize the integrand of \eqref{eq6} while not changing the integration domain, hence transforming the integral into a Kontsevich-Zagier period.
\medskip\par 
We shall show that perhaps surprisingly, the motivic family \eqref{eq2} is intimately related to the 3-parameter hypergeometric subfamily where $a_1=a_2$ and $b_1=b_2$, i.e., the 3-parameter family
\begin{equation}\label{eq4}
\int\limits_{[0,1]^2} \frac{x^{m-1/2}y^{m}(1-x)^{n}(1-y)^{n-1/2}}{(1-xy)^{p+1}} \dd x \dd y \, \quad\quad(m,n,p\in\Z; \ m,n\geq 0; \ 2n\geq p ).
\end{equation}

Our main result shows that on two special 2-parameter subfamilies, the motivic integrals and the hypergeometric integrals coincide up to explicit Gamma-factors. More explicitly, we will prove the following:

\begin{thm}\label{thm1}
\begin{enumerate}[label=(\alph*)]
\item Let $m,n$ be integers satisfying $0\leq m\leq 2n$. Then
\begin{equation}\label{eq1thm}
\int\limits_\Delta\frac{f^nh^m}{g^{2m+1}}\,\dd x\,\dd y
={} 2^{-2n-3}
\frac{\Gamma(n+\frac12)^2}
{\Gamma(2n-m+\frac12)\Gamma(m+\frac12)}
\int\limits_{[0,1]^2}
\frac{x^{m-\frac12}y^m(1-x)^n(1-y)^{n-\frac12}}
{(1-xy)^{m+1}}\,\dd x\,\dd y.
\end{equation}
\item Let $m,n$ be integers satisfying $0\leq m\leq 2n+1$. Then
\begin{equation}\label{eq2thm}
\int\limits_\Delta\frac{f^nh^m}{g^{2m}}\,\dd x\,\dd y
={}2^{-2n-3}
\frac{\Gamma(n+\frac12)\Gamma(n+\frac32)}
{\Gamma(2n-m+\frac32)\Gamma(m+\frac12)}
\int\limits_{[0,1]^2}
\frac{x^{m-\frac12}y^m(1-x)^n(1-y)^{n-\frac12}}
{(1-xy)^m}\,\dd x\,\dd y.
\end{equation}
\end{enumerate}
\end{thm}
The conditions $m\leq 2n$ and $m\leq 2n+1$ in each part are needed for convergence of the hypergeometric side. The $\Gamma$-factor in each equation is a rational number.
\medskip\par 
I have not been able to find a direct proof of Theorem \ref{thm1} by a simple change of variables (or a sequence thereof). The proof we shall give establishes some parallel recurrences for the two families (see \S \ref{sec: I}, \S \ref{sec: J} and \S \ref{subsec: nontrivial base case}). These recurrences, which might be of interest in their own right, are then used together with some base cases to deduce Theorem \ref{thm1} inductively. All the recurrences proved in the paper are established essentially by computations in the relative de Rham cohomology, or more concretely, by applications of Stokes' theorem or linearity with respect to integrands. So the proof of Theorem \ref{thm1} is consistent with the philosophy of the Kontsevich--Zagier period conjecture \cite{KZ}.  
\medskip\par 
Both families \eqref{eq2} and \eqref{eq4} satisfy 2-dimensional Pascal-like recurrences (see \S \ref{subsec: pascal I} and \S \ref{subsec: pascal J}, also the notation of \S\ref{sec: Setup}). This allows one to propagate Theorem \ref{thm1} to relate the two 3-parameter families \eqref{eq2} and \eqref{eq4}. However, the relations one gets from this approach are not as direct as Theorem \ref{thm1} and are not included in the paper.
\medskip\par 
A word on the arithmetic side of the picture is in order. In the special case when $n=m$, the formulas of Theorem \ref{thm1} simply read
\begin{align}
\int\limits_\Delta\frac{h^mf^m}{g^{2m+1}}\,\dd x\,\dd y
&= 2^{-2m-3}
\int\limits_{[0,1]^2}
\frac{x^{m-\frac12}y^m(1-x)^m(1-y)^{m-\frac12}}
{(1-xy)^{m+1}}\,\dd x\,\dd y \label{eq5}\\
\int\limits_\Delta\frac{h^mf^m}{g^{2m}}\,\dd x\,\dd y
&= 2^{-2m-3}
\int\limits_{[0,1]^2}
\frac{x^{m-\frac12}y^m(1-x)^m(1-y)^{m-\frac12}}
{(1-xy)^m}\,\dd x\,\dd y,\notag
\end{align}
which are in $\Q+\Q G$ by any of \cite{RZ}, \cite{Ne} or \cite{EMN}. The subfamily of integrals \eqref{eq5}, which was first discovered by Rivoal and Zudilin \cite{RZ} using a hypergeometric approach, seems to be rather special. On the one hand, Nesterenko \cite{Ne} showed that the integrals \eqref{eq5} lead to\footnote{Denoting the value of the integral \eqref{eq5} by $u_mG-v_m$ with $u_m,v_m\in\Q$, the sequence of rational numbers in question is $(v_m/u_m)$.} a sequence of rational numbers $p_m/q_m$ ($p_m,q_m$ integers $>0$) such that $|G-p_m/q_m|\leq 1/q_m^{0.52}$ for all sufficiently large $m$ (see also \cite{Ri1}, \cite{KR}, \cite{RZ}). This was recently improved in \cite{Es1}, where we used particular linear combinations of integrals in the same family \eqref{eq5} to construct a sequence of rational numbers $p_m/q_m$ ($p_m,q_m$ integers $>0$) such that $|G-p_m/q_m|\leq 1/q_m^{0.62}$ for all sufficiently large $m$. On the other hand, Zudilin has shown in \cite{Zu1} that the sequence of integrals \eqref{eq5} satisfies an Ap\'ery-like recurrence of order two. In \cite{Es3} we show that the generating function of this sequence satisfies a Picard--Fuchs differential equation. It is perhaps interesting that it is exactly at this subfamily that the relation between the larger motivic and hypergeometric families becomes the simplest. We should also mention that Nesterenko\footnote{Note that our notation in \eqref{eq6} is different from Nesterenko's.} \cite{Ne} has given upper bounds for the lowest denominators of the coefficients of 1 and $G$ in the hypergeometric integrals \eqref{eq6}. These bounds seem to be the sharpest on the subfamily \eqref{eq5}.
\medskip\par
After the completion of the present work, we found a simple 3-parameter identity generalizing Theorem \ref{thm1}(a). The proof we have of that identity is of a substantially different nature from the arguments in this paper, relying on classical hypergeometric transformations, and will be treated separately. We have decided to keep the present work in its current form because it also establishes recurrence relations linking the two families that are not apparent from the hypergeometric argument.

\section{Setup}\label{sec: Setup}
Throughout the paper, $\D$ is as in \eqref{eq1} and $f,g,h$ are as in \eqref{eq3}. The letters $m,n,s$ will always refer to integers. The symbol $\epsilon$ will always be $0$ or $1$. It will be convenient to renormalize and reparametrize the family of integrals \eqref{eq2} and \eqref{eq4}, as follows.

For any integers $s,m,n$ and $\epsilon\in \{0,1\}$, we shall set
\begin{equation}\label{eq: basic integral parametrizations 1}
I_\epsilon(s,m,n)
:=\int\limits_\Delta
\left(\frac{4f}{g^2}\right)^{\!s}
\left(\frac{h}{g^2}\right)^{\!m} f^n \,
\frac{\dd x\,\dd y}{g^\epsilon} = 4^s\int\limits_\Delta
\frac{f^{s+n}h^m}{g^{2s+2m+\epsilon}}
\,
\dd x\,\dd y.
\end{equation}
The convergence conditions are
\begin{equation}\label{eq: parameter conds}
m,n, s+n\geq 0.
\end{equation}
Note that the parity of the exponent of $g$ is given by $\epsilon$. The family of integrals we obtain as $\epsilon, m,n,s$ vary with \eqref{eq: parameter conds} are exactly those in \eqref{eq2} with \eqref{eq: original I family convergence cond}. The integrals \eqref{eq2} on the plane $k=2m-1+\epsilon$ (which are the subject of parts (a) or (b) of Theorem \ref{thm1} depending on whether $\epsilon=1$ or $\epsilon=0$, respectively), become the family of integrals $I_\epsilon(s, m,n)$ on the plane $s=0$.

Next, for any integers $s,m,n$, we shall set
\begin{equation}\label{eq: basic integral parametrizations 2}
\begin{split}
    J(s,m,n)
&:=\int\limits_{[0,1]^2}
(1-xy)^s(xy)^m
\left(\!\frac{(1-x)(1-y)}{1-xy}\!\right)^{\!n}
\frac{\dd x\,\dd y}{\sqrt{x(1-y)}(1-xy)}\\
&=\int\limits_{[0,1]^2} \frac{x^{m-1/2}y^m(1-x)^n(1-y)^{n-1/2}}{(1-xy)^{n-s+1}} \dd x\,\dd y.
\end{split}
\end{equation}
The convergence conditions are exactly those in \eqref{eq: parameter conds}. The family of integrals $J(s,m,n)$ as $s,m,n$ vary with \eqref{eq: parameter conds} are exactly those in \eqref{eq4}. The planes $p=m$ and $p=m-1$ (which are the subject of part (a) and (b) of Theorem \ref{thm1}, respectively) correspond to the planes $s=n-m$ and $s=n-m+1$ in the parameters $(s,m,n)$.

It will be also convenient to set
\begin{equation}\label{eq: Gamma factors}
    A_{\epsilon}(m,n)
:=2^{-2n-3}
\frac{\Gamma(n+\frac12)\Gamma(n+\frac32 - \epsilon)}
{\Gamma(2n-m+\frac32-\epsilon)\Gamma(m+\frac12)} \quad\quad\quad (\epsilon=0,1),
\end{equation}
where $\Gamma$ is the Gamma function. We can restate Theorem \ref{thm1} in our new notation as:
\begin{thm}\label{thm1 reformulated}
Let $n,m$ be non-negative integers and $\epsilon\in\{0,1\}$. Then
\[
I_\epsilon(0,m,n)=A_{\epsilon}(m,n)\cdot J(n-m+1-\epsilon,m,n) \qquad (0\leq m\leq 2n+1-\epsilon).
\]
\end{thm}

The case $\epsilon=1$ (resp. $\epsilon=0$) is part (a) (resp. part (b)) of Theorem \ref{thm1}.

\section{Recurrences for the motivic family}\label{sec: I}
\subsection{Pascal relation}\label{subsec: pascal I}
Note that
\begin{equation}\label{eq: h=g^2-4f}
    h= g^2-4f.
\end{equation}
Writing
\[
1 = \frac{h}{g^2}+ \frac{4f}{g^2},
\]
we immediately obtain from \eqref{eq: basic integral parametrizations 1} that for all $\epsilon,m,n,s$ in the convergence range,
\begin{equation}\label{eq: Pascal for I}
I_\epsilon(s,m,n) = I_\epsilon(s+1,m,n) + I_\epsilon(s,m+1,n).
\end{equation}
That is, for each $\epsilon,n$, the 2-dimensional sequence $X(s,m)=I_\epsilon(s,m,n)$ is a solution to the recurrence
\begin{equation}\label{eq: Pascal eq}
X(s,m) = X(s+1,m) + X(s,m+1),
\end{equation}
which we shall call the \emph{Pascal relation}. 

\subsection{Pole reduction recurrences}
We now give recurrences in the motivic family given by a pole reduction process. By the pole order of $I_\epsilon(s,m,n)$ (or more precisely, of the corresponding integrand or differential form) we mean $2s+2m+\epsilon$.

\begin{prop}\label{prop: pole reduction I}
Let $m\ge1$, $n\ge0$, and $s+n\ge0$. Then
\begin{align}
    2(s+m)I_1(s,m,n)+(2n+1)I_0(s,m,n) -2mI_1(s,m-1,n)&=0\label{eq1 pole red I}\\
    (2s+2m-1)I_0(s,m,n)+8(n+1)I_1(s-1,m,n+1)-2mI_0(s,m-1,n) &=0.\label{eq2 pole red I}
\end{align}
\end{prop}

\begin{proof}
Let $\mathbf{E}$ be the differential operator $x\,\partial/\partial x+y \, \partial/\partial_y$. Then for every $F\in \Q(x,y)$,
\begin{equation}\label{eq Euler op}
\dd\,(F\cdot(x\dd y-y\dd x)) =  (\bE F +2F) \,\dd x\, \dd y,
\end{equation}
where here and throughout the paper, to shorten the notation we drop the $\wedge$ symbol for wedge product of differential forms from our writing. Setting $F=f^bh^a/g^c$, a direct computation using 
\[
\bE f=4f,\qquad \bE g=2g-2,\qquad \bE h= \bE (g^2-4f)=4h-4g
\]
and \eqref{eq Euler op} gives 
\[
\dd\left(\frac{f^bh^a}{g^c}(x\dd y-y\dd x)\right)=\left(2c\frac{f^bh^a}{g^{c+1}}+2(2a+2b-c+1)\frac{f^bh^a}{g^c}-4a \frac{f^bh^{a-1}}{g^{c-1}} \right) \dd x\, \dd y.
\]
Now set
\[
\eta
= 4^s\frac{f^{n+s}h^m}{g^{2s+2m-\epsilon}}(x\dd y-y \dd x).
\]
By the last computation we have
\begin{equation}\label{equality of diff forms for pole reduction recursion I family}
\begin{split}
\frac12\,\dd\eta
={}&(2s+2m-\epsilon) 4^s\frac{f^{n+s}h^m}{g^{2s+2m-\epsilon+1}}\,\dd x\,\dd y \ +(2n+1+\epsilon)4^s\frac{f^{n+s}h^m}{g^{2s+2m-\epsilon}}\,\dd x\,\dd y\\
&-2m4^s\frac{f^{n+s}h^{m-1}}{g^{2s+2m-\epsilon-1}}\,\dd x\,\dd y.
\end{split}
\end{equation}
Integrating over the simplex $\D$, the right hand side becomes the left hand sides of \eqref{eq1 pole red I} and \eqref{eq2 pole red I} for $\epsilon=0$ and $\epsilon=1$, respectively (the factor $8$ in the second term of \eqref{eq2 pole red I} comes from the necessary shift in the $s$-coordinate). Thus the proposition is equivalent to the integral of $\dd \eta$ over $\D$ being zero.

To show that $\int_\Delta \dd\eta$ vanishes, we want to apply Stokes' theorem. However, $\dd\eta$ is only regular in $\R^2\setminus A$, where $A$ is the unit circle $g=0$, and $A$ meets $\D$ at the two points $P=(1,0)$ and $Q=(0,1)$ on the boundary $\partial \D$. The remedy is standard, as follows. Here and elsewhere, given two functions $\phi$ and $\psi$ on some set $U\subset \R$ which will be understood from the context, the notation $\phi\ll \psi$ (resp. $\phi\asymp \psi$) means $|\phi|\leq  C|\psi|$ on $U$ for some positive constant $C$ (resp. $\phi \ll \psi$ and $\psi\ll \phi$).

\begin{figure}[ht]
\centering
\begin{tikzpicture}[scale=2.5]
\def\e{0.25}
\fill[gray!15]
  (0,0)
  -- ({1-\e},0)
  -- ({1-\e/2},{\e/2})
  -- ({\e/2},{1-\e/2})
  -- (0,{1-\e})
  -- cycle;
\draw[dashed] (0,0) -- (1,0) -- (0,1) -- cycle;
\draw[thick]
 (0,0)
  -- ({1-\e},0)
  -- ({1-\e/2},{\e/2})
  -- ({\e/2},{1-\e/2})
  -- (0,{1-\e})
  -- cycle;
\draw[dashed]
  (1,0) arc[start angle=0,end angle=90,radius=1];
\fill (1,0) circle (0.012);
\fill (0,1) circle (0.012);
\node[right] at (1,0) {$P=(1,0)$};
\node[above] at (0,1) {$Q=(0,1)$};
\node[above] at ({1-1.1*\e},0.015) {$P_\varepsilon$};
\node[right] at (0.01,{1-1.1*\e}) {$Q_\varepsilon$};
\node at (0.31,0.31) {$\Delta_\varepsilon$};
\node[above right] at (0.70,0.72) {$g=0$};
\node[below] at (0,0) {$(0,0)$};
\end{tikzpicture}
\caption{The truncated simplex $\Delta_\varepsilon$.}
\label{fig:truncated-simplex}
\end{figure}
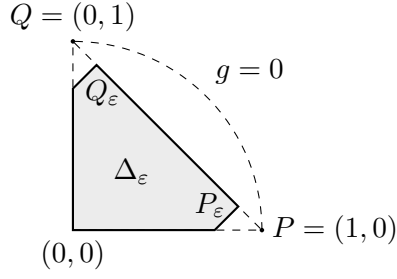
 
For a small $\varepsilon>0$, let $\D_\varepsilon\subset\D$ be the chain obtained by cutting the corners of $\D$ near $P$ and $Q$ as shown in Figure \ref{fig:truncated-simplex}, where the small path marked as $P_\varepsilon$ (resp. $Q_\varepsilon$) is the line segment from $(1-\varepsilon,0)$ to $(1-\varepsilon/2,\varepsilon/2)$ (resp. from $(\varepsilon/2,1-\varepsilon/2)$ to $(0,1-\varepsilon)$). Note that $\eta$ vanishes on the components of $\partial \D_\varepsilon$ on the $x,y$-axes and the line $1-x-y=0$ (the former thanks to $x\dd y-y\dd x$ and the latter because $m\geq 1$). By Stokes' theorem now we get
\[
\int\limits_{\D_\varepsilon} \dd \eta=\int\limits_{\partial\D_\varepsilon} \eta = \int\limits_{P_\varepsilon} \eta+\int\limits_{Q_\varepsilon} \eta.
\]
The integrals of $\eta$ over $P_\varepsilon$ and $Q_\varepsilon$ go to zero as $\varepsilon\rightarrow 0$. Indeed, focusing on the integral over $P_\varepsilon$ as the other one follows by symmetry, parametrize $P_\varepsilon$ by $x=y+(1-\varepsilon)$ as $y$ goes from zero to $\varepsilon/2$, so that $x\dd y-y\dd x=(1-\varepsilon)\dd y$. Using the estimates $0\leq f,h\ll \varepsilon^2$ and $g\asymp \varepsilon$ on $P_\varepsilon$, we get
\[
\int_{P_\varepsilon} \eta \,= 4^s\int_0^{\varepsilon/2} \frac{f^{n+s}h^m}{g^{2s+2m-\epsilon}} (1-\varepsilon) \,\dd y  \ll \varepsilon^{2n+\epsilon+1} \rightarrow 0
\]
as $\varepsilon\rightarrow 0$.
\end{proof}

\begin{rem}\label{rem: family I pole reduction recurrence dR cohomology}
The proof can be formulated in the algebro-geometric language of \cite{EMN}. Indeed, equation \eqref{equality of diff forms for pole reduction recursion I family} implies that the right hand side vanishes in the relative algebraic de Rham cohomology
\[
H^2_\dR(\wX\setminus \wA, (\wB\cup E)\setminus\wA),
\]
where $\wX$ is the blow-up of the affine plane $\mathrm{Spec}\,\Q[x,y]$ at $P$ and $Q$, $E$ is the exceptional divisor above $\{P,Q\}$, $\wA$ is the strict transform of the unit circle $\{g=0\}\subset \mathrm{Spec}\,\Q[x,y]$, and $\wB$ is the strict transform of the boundary divisor $\{xy(1-x-y)=0\}\subset \mathrm{Spec}\,\Q[x,y]$.
\end{rem}

We will be using the pole reduction formulas in an inductive argument to prove Theorem \ref{thm1 reformulated}. Recall that Theorem \ref{thm1 reformulated} concerns the values of the $I$-family when $s=0$. Setting $s=0$ in the first pole reduction formula \eqref{eq1 pole red I}, we obtain a recursion that lives entirely on the parameter plane $s=0$. However, the second formula does not live on a plane with a fixed $s$-value. But this is easily fixed using the Pascal relation (see \S \ref{subsec: pascal I}). This is done in the next corollary, where we also record the specialization of \eqref{eq1 pole red I} to $s=0$ for future reference.

\begin{cor}\label{cor: s=0 I family pole reduction recurrences}
Let $m\ge1$ and $n\ge0$. Then we have
\begin{align}
(2n+1)I_0(0,m,n) &= 2mI_1(0,m-1,n) - 2mI_1(0,m,n)\label{eq1 s=0 pole red I}\\
\begin{split}(2m+1)I_0(0,m+1,n)
&= (4m+1)I_0(0,m,n)-2mI_0(0,m-1,n)\\
& \ \ \ +8(n+1)I_1(0,m,n+1).
\end{split}\label{eq2 s=0 pole red I}
\end{align}
\end{cor}
\begin{proof}
For the first formula, simply set $s=0$ in \eqref{eq1 pole red I}. For the second formula, first set $s=1$ in \eqref{eq2 pole red I}:
\[
(2m+1)I_0(1,m,n)+8(n+1)I_1(0,m,n+1)-2mI_0(1,m-1,n) =0.
\]
By the Pascal relation, we have
\[
\begin{split}
I_0(1,m,n) & = I_0(0,m,n) - I_0(0,m+1,n) \\
I_0(1,m-1,n) & = I_0(0,m-1,n) - I_0(0,m,n).  
\end{split}
\]
Substituting these in the previous equation we get \eqref{eq2 s=0 pole red I}.
\end{proof}

\section{Recurrences for the hypergeometric family}\label{sec: J}
Recall that the hypergeometric family $J(s,m,n)$ is as in \eqref{eq: basic integral parametrizations 2}. The integrals converge if $m,n,s+n\geq 0$.

\subsection{Pascal relation}\label{subsec: pascal J}
Using $1=(1-xy)+xy$ in the defining integral of $J(s,m,n)$ in \eqref{eq: basic integral parametrizations 2} we immediately obtain
\begin{equation}\label{eq J Pascal}
J(s,m,n) = J(s+1,m,n) + J(s,m+1,n) 
\end{equation}
for all integers $m,n,s$ with $m,n,s+n\geq 0$. Thus, fixing $n$, each sequence $X(s,m)=J(s,m,n)$ satisfies (what we called) the Pascal relation
\[
X(s,m) = X(s+1,m) + X(s,m+1).
\]
(Recall from \S \ref{subsec: pascal I} that the $I$-integrals satisfy the same 2-dimensional recurrence.)

\subsection{Pole reduction recurrences}\label{subsec: J pole reduction relations}
The formulas of Theorem \ref{thm1 reformulated} relate 
\[
I_\epsilon(0,m,n) \quad \leftrightarrow \quad J(n-m+1-\epsilon,m,n).
\]
Corollary \ref{cor: s=0 I family pole reduction recurrences} gives two relations within the family of integrals $I_\epsilon(0,m,n)$. In this subsection, we shall give the $J$-analogues of these relations, which will be relations within the family of integrals of the form $J(n-m,m,n)$ and $J(n-m+1,m,n)$. The relations are recorded in Corollary \ref{cor: J special plane relations} below. The four relations in Corollary \ref{cor: s=0 I family pole reduction recurrences} and Corollary \ref{cor: J special plane relations} will play a key role in the proof of Theorem \ref{thm1 reformulated}.

We will obtain the relations in Corollary \ref{cor: J special plane relations} as special cases of two more general recurrences in the $J$-family, which we shall prove first in Propositions \ref{prop:J1} and \ref{prop:J2} below.

\begin{prop}\label{prop:J1}
Let $m\geq 1$, $n\geq 0$ and $s+n\geq 0$. Then
\begin{equation}\label{eq:J1}
\begin{split}
(2m+2s+1)^2J(s+1,m,n) & = 
2m(2m-1)J(s+1,m-1,n)\\
& \quad -2(n-s)(2n+2s+1)J(s,m,n).
\end{split}
\end{equation}
\end{prop}

\begin{proof}
Let
\[
F=
\frac{x^{m-\frac12}y^m(1-x)^n(1-y)^{n-\frac12}}
{(1-xy)^{n-s}}
\]
so that the integral of $F$ over $[0,1]^2$ is $J(s+1,m,n)$. Let $A,B\in\Q(x,y)$ (to be chosen soon) and set
\[
\eta = F\cdot(A\,\dd y -B\,\dd x).
\]
Then
\begin{equation}\label{eq10}
\dd\eta = (F_xA+FA_x+F_y B+FB_y) \, \dd x\, \dd y = (A\frac{F_x}{F}+A_x+B\frac{F_y}{F}+B_y)\, F \,\dd x\,\dd y,
\end{equation}
where the subscripts $x,y$ refer to the partial derivatives. Moreover,
\[
\frac{F_x}{F} = \frac{m-1/2}{x}-\frac{n}{1-x}+\frac{(n-s)y}{1-xy}\,, \qquad \frac{F_y}{F} = \frac{m}{y}-\frac{n-1/2}{1-y}+\frac{(n-s)x}{1-xy}.
\]
The relation \eqref{eq:J1} will be proved if there are $A$ and $B$ such that
\begin{equation}\label{eq8}
\int_{[0,1]^2} \dd\eta =0
\end{equation}
and
\begin{equation}\label{eq7}
\dd\eta = \left(-(2m+2s+1)^2 + \frac{2m(2m-1)}{xy} -\frac{2(n-s)(2n+2s+1)}{1-xy}\right) F \, \dd x\, \dd y,
\end{equation}
as then integrating \eqref{eq7} over $[0,1]^2$ we get the desired formula (see \eqref{eq: basic integral parametrizations 2}). Now set
\[
A = \frac{4(1-x)}{y(1-xy)}(m+(n-m-s)xy), \quad\quad B=\frac{2(1-y)(2m+2n+1-(2m+2s+1)xy)}{1-xy}.
\]
\begin{SaveVerbatim}{sagecheck1}
var('x y m n s')
Fx_over_F = (m-1/2)/x - n/(1-x) + (n-s)*y/(1-x*y)
Fy_over_F = m/y - (n-1/2)/(1-y) + (n-s)*x/(1-x*y)
A = 4*(1-x)*(m + (n-s-m)*x*y)/(y*(1-x*y))
B = 2*(1-y)*(2*m+2*n+1-(2*m+2*s+1)*x*y)/(1-x*y)
LHS = diff(A,x) + A*Fx_over_F + diff(B,y) + B*Fy_over_F
RHS = (2*m*(2*m-1)/(x*y)
       - 2*(n-s)*(2*n+2*s+1)/(1-x*y)
       - (2*m+2*s+1)^2)
print((LHS-RHS).simplify_full())
\end{SaveVerbatim}
The reader can verify that\footnote{For convenience, the reader can use (copy and paste) the SageMath code below to verify this calculation. The code computes the difference of the two sides of the equation (symbolically).
\UseVerbatim[fontsize=\tiny]{sagecheck1}}
\[
    A\frac{F_x}{F}+A_x+B\frac{F_y}{F}+B_y =-(2m+2s+1)^2 + \frac{2m(2m-1)}{xy} -\frac{2(n-s)(2n+2s+1)}{1-xy},
\]
so that \eqref{eq7} indeed holds.
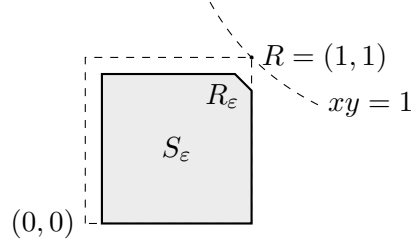
\begin{figure}[ht]
\centering
\begin{tikzpicture}[scale=2.2]
\def\e{0.20}
\fill[gray!15]
  ({\e/2},0)
  -- (1,0)
  -- (1,{1-\e})
  -- ({1-\e/2},{1-\e/2})
  -- ({\e/2},{1-\e/2})
  -- cycle;
\draw[dashed] (0,0) rectangle (1,1);
\draw[dashed, domain=0.75:1.4, samples=100, smooth]
  plot (\x,{1/\x});
\node[right] at (1.4,1/1.4) {$xy=1$};
\draw[thick]
  ({\e/2},0)
  -- (1,0)
  -- (1,{1-\e})
  -- ({1-\e/2},{1-\e/2})
  -- ({\e/2},{1-\e/2})
  -- cycle;
\fill (1,1) circle (0.012);
\node[right] at (1,1) {$R=(1,1)$};
\node at (0.55,0.45) {$S_\varepsilon$};
\node[below left]
  at ({1-\e/16},{1-\e/2})
  {$R_\varepsilon$};
\node[left] at (0,0) {$(0,0)$};
\end{tikzpicture}
\caption{The truncated square $S_\varepsilon$.}
\label{fig:truncated-square}
\end{figure}

As for \eqref{eq8}, choose small $\varepsilon>0$. Let $S_{\varepsilon}$ be the chain obtained from $[\varepsilon/2,1]\times [0,1-\varepsilon/2]$ by cutting the corner near $R=(1,1)$ (see Figure \ref{fig:truncated-square}). The small corner path $R_\varepsilon$ goes from $(1,1-\varepsilon)$ to $(1-\varepsilon/2,1-\varepsilon/2)$ on the line $x+y=2-\varepsilon$. Then $\eta$ is a smooth differential form in a neighbourhood of $S_\varepsilon$ (note that the pole $y$ of $A$ is cancelled by the factor $y^m$ of $F$ because $m\geq 1$), so we can use Stokes' theorem to compute the integral of $\dd \eta$ on $S_{\varepsilon}$. We have
\[
\int_{[0,1]^2} \dd \eta = \lim_{\varepsilon\rightarrow 0}  \int_{S_{\varepsilon}} \dd \eta = \lim_{\varepsilon\rightarrow 0} \int_{\partial S_{\varepsilon}} \eta.
\]
(For the first equality, we already know from \eqref{eq7} that the integral of $\dd \eta$ over $[0,1]^2$ converges, because $m\geq 1$ and $n, n+s\geq 0$.) The integrals of $\eta$ over the bottom horizontal and the right vertical edges of $\partial S_\varepsilon$ are zero as $\eta$ vanishes there (respectively, thanks to $m\geq 1$ and the zero $1-x$ of $A$). We will shows that as $\varepsilon\rightarrow 0$, the integrals of $\eta$ along the other three edges of $\partial S_\varepsilon$ tend to zero. Indeed, on the left vertical edge we have $\eta= FA\dd y$. First using $A\ll (1-x)/(y(1-xy))$ and then $y\ll 1$, $1-x\ll 1$ and $1-xy \asymp 1$ on that edge, on recalling $m\geq 1$ we get $FA \ll x^{m-1/2}(1-y)^{n-1/2}\ll \varepsilon^{m-1/2}(1-y)^{n-1/2}$, so that the integral of $\eta$ over the left vertical edge is
\[
\ll \varepsilon^{m-1/2}\int_0^1 (1-y)^{n-1/2} \dd y \rightarrow 0
\]
as $m\geq 1$ and $n\geq 0$. Moving on to the small corner path $R_\varepsilon$, we have $A,B\ll 1$ and $1-xy\asymp \varepsilon$ there. Setting $u=1-y$ so that $1-x=\varepsilon-u$, we get
\[
\int_{R_\varepsilon} \eta \, \ll \varepsilon^{-(n-s)}\int_{0}^\varepsilon (\varepsilon-u)^nu^{n-1/2}\dd u \ll \varepsilon^{s}\int_0^\varepsilon u^{n-1/2}\dd u=\frac{1}{n+1/2} \varepsilon^{n+s+1/2} \rightarrow 0
\]
as $\varepsilon\rightarrow 0$ because $n+s\geq 0$.

It remains to treat the integral of $\eta$ over the top horizontal edge of $\partial S_\varepsilon$. On this edge, $y=1-\varepsilon/2$, $\eta=-FB\dd x$ and
\[
FB \ll \frac{x^{m-1/2}y^m(1-x)^n(1-y)^{n+1/2}}{(1-xy)^{n-s+1}} \ll \frac{\varepsilon^{n+1/2} x^{m-1/2}(1-x)^n}{(1-xy)^{n-s+1}}.
\]
We have both $1-xy\gg \varepsilon$ and $1-xy\asymp 1-x$ on this edge. Using the former on $(1-xy)^n$ and the latter on $(1-xy)^{-s+1}$ we get $FB \ll \varepsilon^{1/2}x^{m-1/2}(1-x)^{n+s-1}$. Thus the integral of $\eta$ over the top horizontal edge is
\[
\ll \varepsilon^{1/2}\int_{\varepsilon/2}^{1-\varepsilon/2}
x^{m-1/2}(1-x)^{n+s-1}\dd x \ll \varepsilon^{1/2}\log\epsilon \rightarrow 0
\]
as $\varepsilon\rightarrow 0$ because $m,s+n\geq 0$.\footnote{Note that we did not use $m\geq 1$ in our argument for the integral of $\eta$ on the top horizontal edge going to zero. This will be relevant later in the proof of Proposition \ref{prop: base case e=1, m=0}.}
\end{proof}

\begin{prop}\label{prop:J2}
Let $m\geq 1$, $n\geq 0$ and $s+n\geq 0$. Then
\begin{equation}\label{eq:J2}
\begin{split}
2(n+1)(2n+1)J(s+1,m,n+1) = &\ 2m(2m-1)J(s+2,m-1,n) \\
&-(4m+1)(2n+2s+3)J(s+1,m,n)\\
&+(2n+2s+1)(2n+2s+3)J(s,m+1,n).
\end{split}
\end{equation}
\end{prop}

\begin{proof}
Take
\begin{align*}
A&=
\frac{2(1-x)}{y(1-xy)}\left((2m-(2m+2s+2)xy-(2n+1)xy^2\right)\\
B&=\frac{2(1-y)}{1-xy}
\left(2m+2n+1-(2n+1)y-(2m+2n+2s+3)xy\right)
\end{align*}
\begin{SaveVerbatim}{sagecheck2}
var('x y m n s')
Fx_over_F = (m-1/2)/x - n/(1-x) + (n-s)*y/(1-x*y)
Fy_over_F = m/y - (n-1/2)/(1-y) + (n-s)*x/(1-x*y)
A = 2*(1-x)/(y*(1-x*y)) *(2*m-(2*m+2*s+2)*x*y-(2*n+1)*x*y^2)
B = 2*(1-y)/(1-x*y) * (2*m+2*n+1-(2*n+1)*y-(2*m+2*n+2*s+3)*x*y)
LHS = diff(A,x) + A*Fx_over_F + diff(B,y) + B*Fy_over_F
RHS = (2*m*(2*m-1)*(1-x*y)/(x*y)-(4*m+1)*(2*n+2*s+3)
+(2*n+2*s+1)*(2*n+2*s+3)*x*y/(1-x*y)-2*(n+1)*(2*n+1)*(1-x)*(1-y)/(1-x*y))
print((LHS-RHS).simplify_full())
\end{SaveVerbatim}
and set $\eta=F(A\,\dd y-B\,\dd x)$, where $F$ is as in the proof of Proposition \ref{prop:J1}. A direct computation gives\footnote{For convenience, here is a SageMath code that verifies this calculation. The code computes the difference of the two sides of the equation.
\par\medskip
\UseVerbatim[fontsize=\tiny]{sagecheck2}}
\[\begin{split}
A\frac{F_x}{F}+A_x+B\frac{F_y}{F}+B_y =& \ 2m(2m-1)\frac{1-xy}{xy}
-(4m+1)(2n+2s+3)\\
&+(2n+2s+1)(2n+2s+3)\frac{xy}{1-xy}\\
&-2(n+1)(2n+1)\frac{(1-x)(1-y)}{1-xy}.
\end{split}
\]
Thus by \eqref{eq10},
\begin{equation}\label{eq11}
\begin{split}
\dd\eta
={}&
\Bigm(2m(2m-1)\frac{1-xy}{xy}
-(4m+1)(2n+2s+3)\\
&+(2n+2s+1)(2n+2s+3)\frac{xy}{1-xy}\\
&-2(n+1)(2n+1)\frac{(1-x)(1-y)}{1-xy}\Bigm) F \, \dd x\, \dd y.
\end{split}
\end{equation}
The four terms on the right hand side (after taking the factor $F$ into account) are scalar multiples of the integrands of
\[
J(s+2,m-1,n),\quad
J(s+1,m,n),\quad
J(s,m+1,n),\quad
J(s+1,m,n+1)
\]
(note that multiplying the integrand of a $J$-integral by
$(1-x)(1-y)/(1-xy)$ simply increases the parameter $n$ by 1). Integrating \eqref{eq11} over $[0,1]^2$ we get \eqref{eq:J2}, as the integral of $d\eta$ is again zero by the same considerations as in the proof of Proposition \ref{prop:J1}.
\end{proof}

We now proceed to specialize the recurrences of this subsection to the cases useful for the proof of Theorem \ref{thm1 reformulated}. It will be convenient to set
\begin{equation}\label{eq J special plane notation}
J_\epsilon(m,n) := J(n-m+1-\epsilon, m,n)\qquad \qquad (\epsilon\in\{0,1\}).     
\end{equation}
Thus the $J_\epsilon(m,n)$ are exactly the $J$-integrals that appear in Theorem \ref{thm1 reformulated}.

\begin{cor}\label{cor: J special plane relations}
Let $1\leq m\leq 2n$.
\begin{enumerate}[label=(\alph*)]
\item We have
\[
\begin{split}
(2n+1)^2J_0(m,n) = & \ 2m(2m-1)J_1(m-1,n)
-2m(4n-2m+1)J_1(m,n).
\end{split}
\]
\item We have
\[
\begin{split}
2(n+1)(2n+1)J_1(m,n+1) = & \ 2m(2m-1)J_0(m-1,n)\\
&-(4m+1)(4n-2m+3)J_0(m,n)\\
&+(4n-2m+1)(4n-2m+3)J_0(m+1,n).
\end{split}
\]
\end{enumerate}
\end{cor}

\begin{proof}
For part (a) (resp. part (b)), set $s=n-m$ in Proposition \ref{prop:J1} (resp. Proposition \ref{prop:J2}).    
\end{proof}

\begin{rem}
The relations in Corollary \ref{cor: J special plane relations} live entirely within the family of integrals $J(s,m,n)$ with $s=n-m$ or $s=n-m+1$, which are exactly the subfamily of the $J$-integrals appearing in Theorem \ref{thm1 reformulated}. Neither the relation of Proposition \ref{prop:J1} nor the relation of Proposition \ref{prop:J2} lives entirely within this desired subfamily unless $s=n-m$. 
\end{rem}

\begin{rem}
Propositions \ref{prop:J1} and \ref{prop:J2} can also be derived from the two recurrences in the 5-parameter family \eqref{eq6} obtained by integrating $\dd(x(1-x)\Phi \dd y)$ and $\dd(y(1-y)\Phi \dd x)$ over $[0,1]^2$, where
\[
\Phi=\frac{x^{a_1-1/2}y^{a_2}(1-x)^{b_1}(1-y)^{b_2-1/2}}{(1-xy)^{a_3+1}}
\]
is the integrand of \eqref{eq6}. The arguments we gave for the two propositions are a little shorter.
\end{rem}

\subsection{Relation to the $I$-recurrences}
The two relations in Corollary \ref{cor: J special plane relations} are closely related to the relations of Corollary \ref{cor: s=0 I family pole reduction recurrences}. To make this clear, we shall set
\begin{equation}\label{eq: 2-parameter J renormalized}
\Jh_\epsilon(m,n):= A_\epsilon(m,n) \cdot J_\epsilon(m,n) \qquad (\epsilon\in\{0,1\}, \ 0\leq m\leq 2n+1-\epsilon),
\end{equation}
where $J_\epsilon(m,n)$ is as in \eqref{eq J special plane notation} and $A_{\epsilon}(m,n)$ is the rational Gamma-factor
\begin{equation}\label{eq: recall A notation}
A_{\epsilon}(m,n)
=2^{-2n-3}
\frac{\Gamma(n+\frac12)\Gamma(n+\frac32 - \epsilon)}
{\Gamma(2n-m+\frac32-\epsilon)\Gamma(m+\frac12)} \quad\quad\quad (\epsilon=0,1)
\end{equation}
(as earlier defined in \eqref{eq: Gamma factors}). In other words, $\Jh_\epsilon(m,n)$ is exactly the expression on the right hand side of the equation of Theorem \ref{thm1 reformulated}. 

Corollary \ref{cor: J special plane relations} can be reformulated as follows:

\begin{cor}\label{cor: renormalized J relations}
Let $1\leq m\leq 2n$. Then we have
\begin{align}
(2n+1)\Jh_0(m,n)&= 2m\Jh_1(m-1,n) -2m\Jh_1(m,n)\label{eq1 pole red Jh}.\\
\begin{split}(2m+1)\Jh_0(m+1,n)
&= (4m+1)\Jh_0(m,n)-2m\Jh_0(m-1,n)\\
& \ \ \ +8(n+1)\Jh_1(m,n+1).
\end{split}\label{eq2 pole red Jh}
\end{align}
That is, in the region of convergence, the $\Jh_\epsilon(m,n)$ satisfy the exact same recurrences as those for $I_\epsilon(0,m,n)$ given in Corollary \ref{cor: s=0 I family pole reduction recurrences}.
\end{cor}

\begin{proof}
For the first identity, note that
\[
\frac{A_0(m,n)}{2n+1}=\frac{A_1(m-1,n)}{2m-1}=\frac{A_1(m,n)}{4n-2m+1}.
\]
Now start with the identity of Corollary \ref{cor: J special plane relations}(a), multiply the left hand side by $A_0(m,n)/(2n+1)$, the first expression on the right hand side by $A_1(m-1,n)/(2m-1)$ and the second expression by $A_1(m,n)/(4n-2m+1)$. On recalling the definition of $\Jh_\epsilon(m,n)$ we get exactly \eqref{eq1 pole red Jh}. 

Similarly, the second relation is obtained by multiplying the equation of Corollary \ref{cor: J special plane relations}(b) by
\[
\frac{4A_1(m,n+1)}{2n+1}=\frac{A_0(m-1,n)}{2m-1}=\frac{A_0(m,n)}{4n-2m+3}=\frac{(2m+1)\,A_0(m+1,n)}{(4n-2m+1)(4n-2m+3)}.\qedhere
\]    
\end{proof}

\section{Proof of Theorem \ref{thm1}}\label{sec: thm proof}
For simplicity, let us set
\[
I_\epsilon(m,n) := I_\epsilon(0,m,n) = \int_\D \frac{f^nh^m}{g^{2m+\epsilon}} \dd x \dd y,
\]
where as before, $\D$ is the simplex defined by $x,y,1-x-y\geq 0$ and
\[
    f=x^2y^2 ,\quad\quad g=1-x^2-y^2, \quad\quad h=\prod_{\delta,\delta'\in\{\pm 1\}} (1-\delta x -\delta' y)=g^2-4f.
\]
Also recall that
\[
J_\epsilon(m,n)= J(n-m+1-\epsilon,m,n)=\int_{[0,1]^2} \frac{x^{m-1/2}y^m(1-x)^n(1-y)^{n-1/2}}{(1-xy)^{m+\epsilon}} \dd x\dd y
\]
and
\[
\Jh_\epsilon(m,n) = A_\epsilon(m,n) J_\epsilon(m,n),
\]
where $A_\epsilon(m,n)$ is as in \eqref{eq: recall A notation}. The identity of Theorem \ref{thm1 reformulated} simply reads
\[
I_\epsilon(m,n) = \Jh_\epsilon(m,n) \qquad\quad (\epsilon\in\{0,1\}, \, 0\leq m\leq 2n+1-\epsilon).
\]
We will prove this identity by induction on $m$. The inductive step will essentially be immediate using Corollaries \ref{cor: s=0 I family pole reduction recurrences} and \ref{cor: renormalized J relations}. In what follows, we first establish the result on the required base cases, then move ahead to deduce Theorem \ref{thm1 reformulated}.

\subsection{The case $(\epsilon,m)=(0,0)$}
Here we show that
\[
I_0(0,n) = \Jh_0(0,n)
\]
for all $n\geq 0$. The relevant integrals are
\[
I_0(0,n) = \int_\D x^{2n}y^{2n} \dd x\dd y, \qquad J_0(0,n)=\int_{[0,1]^2}\frac{(1-x)^n(1-y)^n}{\sqrt{x(1-y)}} \dd x \dd y.
\]
Both these integrals are easy to compute. Indeed, writing $x^{2n}y^{2n} \dd x\dd y=\dd (\frac{x^{2n+1}y^{2n}}{2n+1}\dd y)$ and applying Stokes' theorem we see that
\begin{equation}\label{eq12}
I_0(0,n) = \frac{1}{2n+1}\int_0^1 y^{2n}(1-y)^{2n+1}\dd y= \frac{1}{2n+1} B(2n+1,2n+2) = \frac{((2n)!)^2}{(4n+2)!},
\end{equation}
where $B$ is the Beta function (and we used $B(a,b)=\Gamma(a)\Gamma(b)/\Gamma(a+b)$). As for the integral for $J_0(0,n)$, it factors as a product of integrals $\dd x$ and $\dd y$, each of which is a Beta-value:
\begin{equation}\label{eq: J_0(0,n) values}
J_0(0,n) = B(1/2,n+1)\cdot B(1,n+1/2)=\frac{\Gamma(1/2)\Gamma(n+1)}{(n+1/2)\Gamma(n+3/2)}.
\end{equation}
Using this and
\[
A_0(0,n) = 2^{-2n-3}\frac{\Gamma(n+1/2)\Gamma(n+3/2)}{\Gamma(2n+3/2)\Gamma(1/2)},
\]
after simplification and standard manipulations we get $A_0(0,n) J_0(0,n)=((2n)!)^2/(4n+2)!$, which gives us the desired conclusion in view of \eqref{eq12}.

\subsection{The case $(\epsilon,m)=(1,0)$}\label{subsec: nontrivial base case}
In this subsection we will prove the following proposition, which establishes Theorem \ref{thm1 reformulated} in the case where $\epsilon=1$ and $m=0$. 
Before proceeding, for convenience, we note
\[
\begin{split}
I_1(0,n) = \int_\D\frac{f^n}{g}\dd x\dd y,&\qquad J_1(0,n)=\int_{[0,1]^2}\frac{(1-x)^n(1-y)^{n-1/2}}{\sqrt{x} (1-xy)}\dd x\dd y,\\ A_1(0,n)&=2^{-2n-3}\frac{\Gamma(n+1/2)^2}{\Gamma(2n+1/2)\Gamma(1/2)}.    
\end{split}
\]
\medskip\par 
\begin{prop}\label{prop: base case e=1, m=0} ~ 
\begin{enumerate}[label=(\alph*)]
\item The sequences $X(n)=I_1(0,n)$ and $X(n)=\Jh_1(0,n)$ ($n\geq 0)$ satisfy the first order non-homogeneous recurrence
\begin{equation}\label{eq: order 1 recurrence for integral of f^n/g}
8(n+1)X(n+1) = (2n+1) X(n) - \frac{(6n+5)\,(2n)!\,(2n+1)!}{(4n+3)!}.
\end{equation}
\item We have 
\[
I_1(0,n) = \Jh_1(0,n)
\]
for all $n\geq 0$.
\end{enumerate}    
\end{prop}

\begin{proof}
We first note that part (b) is immediate from part (a) and the initial values
\[
I_1(0,0)=\int_\D \frac{\dd x \dd y}{g} = G,\quad J_1(0,0)=\int_{[0,1]^2}\frac{x^{-1/2}(1-y)^{-1/2}}{1-xy} \dd x\dd y = 8G,\quad
A_1(0,0)=1/8.
\]
(For the integral $J_1(0,0)$, see \cite[\S 9]{RZ}, for instance\footnote{To avoid any chance of confusion, we note that there is a factor 2 misprint in this calculation in a preprint version of \cite{RZ} that is readily available online. The misprint is corrected in the published version.}. For a reference for the integral $I_1(0,0)$, see \cite{JL}, equation (20), where a variant of this integral is evaluated. The integral $I_1(0,0)$ is easily derived from the one there by a change of variables.)
\medskip\par 
We now need to prove part (a). We first focus on the sequence $I_1(0,n)=I_1(0,0,n)$. Equation \eqref{equality of diff forms for pole reduction recursion I family} is valid as an equality of meromorphic differential forms for all integers $s,m,n$. Setting $m=0$, $s=1$, and $\epsilon=1$ it reads
\[
\frac12\,\dd\eta
={}4\frac{f^{n+1}}{g^{2}}\,\dd x\,\dd y \ +8(n+1)\frac{f^{n+1}}{g}\,\dd x\,\dd y,
\]
where
\[
\eta = 4\frac{f^{n+1}}{g}(x\dd y-y \dd x).
\]
Integrating over $\D$ we get the non-homogeneous relation
\begin{equation}\label{eq13}
I_0(1,0,n)+8(n+1) I_1(0,0,n+1)=\frac{1}{2}\int_\D \dd\eta  \stackrel{(\ast)}{=} B(2n+2,2n+2).
\end{equation}
(Compare with the second relation of Proposition \ref{prop: pole reduction I}. For $(\ast)$, use Stokes' theorem on the chain $\D_\varepsilon$ as in the proof of Proposition \ref{prop: pole reduction I}. The integrals of $\eta$ on the horizontal and vertical parts of $\partial \D_\varepsilon$ are clearly zero, and the integrals over the small corner paths are easily seen to go to zero as $\varepsilon\rightarrow 0$. The integral of $\eta$ on the straight-line path $PQ$ from $P=(1,0)$ to $Q=(0,1)$ gives the right hand side of $(\ast)$, as $g=2y(1-y)$ and $x\dd y-y\dd x=\dd y$ on $PQ$.)

To show that the sequence $I_1(0,0,n)$ satisfies the recursion \eqref{eq: order 1 recurrence for integral of f^n/g}, we need another relation to eliminate $I_0(1,0,n)$ from \eqref{eq13}. Set
\[
F=\frac{f^n}{g},\quad A=\frac{x(1-x^2+y^2)}{2},\quad B=\frac{y(1+x^2-y^2)}{2}, \quad \eta=F(A\dd y-B\dd x),
\]
(where we have recycled the symbol $\eta$). The reader can verify using
\[
A_x+B_y=g, \quad \frac{F_x}{F}=\frac{\dd}{\dd x}\log(f^n/g)=\frac{2n}{x}+\frac{2x}{g},\quad \frac{F_y}{F}=\frac{2n}{y}+\frac{2y}{g}
\]
and \eqref{eq10} that
\[
d\eta = (2n+1+\frac{4x^2y^2}{g})\, F \,\dd x\,\dd y = (2n+1)\frac{f^n}{g} \dd x\,\dd y+4\frac{f^{n+1}}{g^2}\dd x \,\dd y.
\]
Integrating over $\D$, by similar considerations as before we get
\[
(2n+1)I_1(0,0,n)+I_0(1,0,n)= \int_\D \dd \eta= \int_{PQ}\eta =B(2n+1,2n+1).
\]
(Note that $g=A+B$ on $PQ$.) Subtracting this from \eqref{eq13} after simplification we obtain \eqref{eq: order 1 recurrence for integral of f^n/g} with $X(n)=I_1(0,0,n)$, as desired.

We now turn our attention to the sequence $\Jh_1(0,n)$. With $\eta, A, B$ as in the proof of Proposition \ref{prop:J2}, equation \eqref{eq11} holds for all $s,m,n\in\Z$. Setting $m=0$ and $s=n$, it reads
\begin{equation}\label{eq14}
\begin{split}
\dd\eta
={}&
\Bigm(
-(4n+3)+(4n+1)(4n+3)\frac{xy}{1-xy}\\
&-2(n+1)(2n+1)\frac{(1-x)(1-y)}{1-xy}\Bigm) F \, \dd x\, \dd y,
\end{split}
\end{equation}
where
\[\begin{split}
F = & \ x^{-1/2}(1-x)^n(1-y)^{n-1/2}, \qquad\qquad \eta= F\left( A\dd y-B\dd x\right)\\
A=&\frac{-2x(1-x)}{1-xy}\bigm(\!(2n+2)+(2n+1)y\!\bigm)\\
B=&\frac{2(1-y)}{1-xy}\bigm(\!(2n+1)(1-y)-(4n+3)xy\!\bigm)\!.
\end{split}
\]
We integrate \eqref{eq14} over $[0,1]^2$. As in \S \ref{subsec: J pole reduction relations}, we may apply Stokes' theorem to the integral of $\dd \eta$ over the region $S_\varepsilon$ of Figure \ref{fig:truncated-square} (see the proof of Proposition \ref{prop:J1} to recall the description of $S_\varepsilon$). The differential form $\eta$ vanishes on the right vertical edge of $\partial S_\varepsilon$, and similar considerations to the ones in the proof of Proposition \ref{prop:J1} show that the integral of $\eta$ over the left vertical edge, the small corner path, and the top horizontal edge in $\partial S_\varepsilon$ all tend to zero as $\varepsilon\rightarrow 0$. (Comparing with the situation of Propositions \ref{prop:J1} and \ref{prop:J2} where $m\geq 1$, here on the vertical edges $A$ comes to our rescue. As for the corner edge and the top edge, the estimates in Proposition \ref{prop:J1} did not use the assumption $m\geq 1$.) On the other hand, $\eta$ now no longer vanishes on the $x$-axis; its integral over the bottom horizontal edge of $\partial S_\varepsilon$ tends to
\[
-(4n+2)\int_0^1x^{-1/2}(1-x)^n \dd x=-(4n+2) \, B(1/2, n+1).
\]
(Note that $\eta= -BF\dd x = -(4n+2) x^{-1/2}(1-x)^n \dd x$ on the $x$-axis.)
Thus integrating \eqref{eq14} over $[0,1]^2$, on recalling \eqref{eq: basic integral parametrizations 2} we get
\begin{equation}\label{eq16}
\begin{split}
&-(4n+3) J(n+1,0,n)
+(4n+1)(4n+3)J(n,1,n)\\
&-2(n+1)(2n+1)J(n+1,0,n+1) = -(4n+2) B(1/2,n+1).
\end{split}
\end{equation}
Using the Pascal relation \eqref{eq J Pascal} to replace $J(n,1,n)$ by $J(n,0,n)-J(n+1,0,n)$, in view of \eqref{eq: J_0(0,n) values} (and on recalling the definition of $J_\epsilon(m,n)$ from \eqref{eq J special plane notation}), we get
\begin{equation}\label{eq15}
\begin{split}
(4n+1)(4n+3)J_1(0,n) &- 2(n+1)(2n+1)J_1(0,n+1)= \\
&(4n+2)B(\frac{1}{2},n+1) \left(\!(4n+3)B(1,n+\frac{1}{2})-1\!\right)=\frac{2(6n+5)\Gamma(\frac{1}{2}) \, n!}{\Gamma(n+\frac32)}.
\end{split}
\end{equation}
Straightforward manipulations using
\[
A_1(0,n)=2^{-2n-3}\frac{\Gamma(n+1/2)^2}{\Gamma(2n+1/2)\Gamma(1/2)}=\frac{4(4n+1)(4n+3)}{(2n+1)^2}A_1(0,n+1)
\]
show that this amounts to \eqref{eq: order 1 recurrence for integral of f^n/g} with $X(n)=\Jh_1(0,n)$.
\end{proof}

\subsection{The case $m=1$}
The goal of this subsection is to establish Theorem \ref{thm1 reformulated} in the case when $m=1$, that is to show that
\begin{equation}\label{eq: thm m=1 case}
I_\epsilon(1,n) = \Jh_\epsilon(1,n) \qquad\qquad (\epsilon\in\{0,1\}, \  n\geq \epsilon).
\end{equation}
(The condition $n\geq \epsilon$ comes from the convergence condition $m\leq 2n+1-\epsilon$.)

We can deduce \eqref{eq: thm m=1 case} from the case $m=0$ of Theorem \ref{thm1 reformulated} using our earlier recursions, as follows. Both $X_\epsilon(m,n)=I_\epsilon(m,n)$ and $X_\epsilon(m,n)=\Jh_\epsilon(m,n)$ satisfy the relation
\begin{equation}\label{eq17}
X_0(1,n) = X_0(0,n)+8(n+1)X_1(0,n+1)- B(2n+2,2n+2).
\end{equation}
Taking this for granted for a moment, since we already know that $I_\epsilon(0,n)=\Jh_\epsilon(0,n)$ for all $n\geq 0$ and both $\epsilon=0,1$, it follows that $I_0(1,n)=\Jh_0(1,n)$. As for the relation \eqref{eq17}, if $X_\epsilon(m,n)=I_\epsilon(m,n)$ this follows from \eqref{eq13} upon using the Pascal relation to replace $I_0(1,0,n)$ by $I_0(0,0,n)-I_0(0,1,n)$ (recall that $I_\epsilon(m,n)$ means $I_\epsilon(0,m,n)$). As for $X_\epsilon(m,n)=\Jh_\epsilon(m,n)$, one can see by straightforward manipulations that
\[
\frac{A_0(0,n)}{4n+3}=\frac{4A_1(0,n+1)}{2n+1}=\frac{A_0(1,n)}{(4n+1)(4n+3)}=\frac{B(2n+2,2n+2)}{(4n+2)B(1/2,n+1)}
\]
(see \eqref{eq: recall A notation} to recall the definition of $A_\epsilon(m,n)$). Thus \eqref{eq17} is exactly what \eqref{eq16} gives in view of the normalization $\Jh_\epsilon(m,n)=A_\epsilon(m,n)J_\epsilon(m,n)$ (see also \eqref{eq J special plane notation}).

It remains to establish the case $\epsilon=1$ of \eqref{eq: thm m=1 case}. Setting $m=1$ in the first relations of Corollaries \ref{cor: renormalized J relations} and \ref{cor: s=0 I family pole reduction recurrences}, both $X_\epsilon(m,n)=I_\epsilon(m,n)$ and $X_\epsilon(m,n)=\Jh_\epsilon(m,n)$ satisfy the relation
\[
(2n+1)X_0(1,n) = 2 X_1(0,n) - 2 X_1(1,n)
\]
for all $n\geq 1$. Thus the case $\epsilon=1$ of \eqref{eq: thm m=1 case} follows from the already established cases $m=0$ and $(\epsilon,m)=(0,1)$.

\subsection{Finishing the proof of Theorem \ref{thm1 reformulated}}
We are ready to finish the proof of Theorem \ref{thm1 reformulated}. The goal is to show that for every $m\geq 0$,
\[
I_\epsilon(m,n) = \Jh_\epsilon(m,n) \quad \text{for all $\epsilon\in\{0,1\}$ and $n$ with $m\leq 2n+1-\epsilon$}.
\]
The cases $m=0,1$ have already been established. Arguing by induction on $m$, suppose that the desired conclusion for both $\epsilon=0,1$ holds for all $m\leq m'$ for some $m'\geq 1$. We first prove the statement for $m=m'+1$ and $\epsilon=0$. Fix $n$ with $m'\leq 2n$, so that $n$ is in the range where the theorem applies for $m=m'+1$, $\epsilon=0$. First use the second identities of Corollaries \ref{cor: s=0 I family pole reduction recurrences} and \ref{cor: renormalized J relations} (applied with the choice $m=m'$, referring to the notation of said corollaries): for $X\in\{ I,\Jh\}$, we have
\[
(2m'+1)X_0(m'+1,n)
= (4m'+1)X_0(m',n)-2m' X_0(m'-1,n)+8(n+1)X_1(m',n+1).
\]
The right hand sides coincide for $X=I$ and $X=\Jh$, so that $I_0(m'+1,n)=\Jh_0(m'+1,n)$. Thus the case $\epsilon=0$, $m=m'+1$ of the assertion holds.

Next, we show that the case $\epsilon=1$, $m=m'+1$ also holds. Now fix $n$ with $m'\leq 2n-1$. We apply the first recurrences of Corollaries \ref{cor: s=0 I family pole reduction recurrences} and \ref{cor: renormalized J relations} (applied with the choice $m=m'+1$). For $X\in\{ I,\Jh\}$, we have
\[
(2n+1)X_0(m'+1,n)= (2m'+2) X_1(m',n) -(2m'+2) X_1(m'+1,n).
\]
We already know $I_0(m'+1,n)=\Jh_0(m'+1,n)$ and $I_1(m',n)=\Jh_1(m',n)$, so that we get $I_1(m'+1,n)=\Jh_1(m'+1,n)$, finishing the proof of Theorem \ref{thm1 reformulated}.

\section*{Acknowledgments}
This work is a continuation of the joint work \cite{EMN} with Kumar Murty and Yusuke Nemoto. I would like to express my utmost gratitude to Murty and Nemoto for very helpful discussions over the course of this work.

\end{document}